\documentclass[12pt]{amsart}
\usepackage{graphicx}
\usepackage{subfig}
\usepackage{float}
\usepackage{wasysym}
\usepackage{amssymb}
\usepackage{amsmath,amsthm,enumerate}
\usepackage{mathtools}

\DeclareMathOperator{\RE}{Re} 

 \def \D{\mathbb{D}}
 \def \C{\mathbb{C}}
 \def  \E{e^{i \theta}}

\numberwithin{equation}{section}
\newtheorem{theorem}{Theorem}[section]

\newtheorem{corollary}[theorem]{Corollary}
\theoremstyle{remark}
\newtheorem{remark}[theorem]{Remark}

\allowdisplaybreaks
 
\begin{document}

\title[Differential Subordination for Janowski  Functions]{Differential Subordination for   Janowski  Functions with Positive Real Part}

\author[S. Anand]{Swati Anand}
\address{Rajdhani College, University of Delhi, Delhi--110015, India}
\email{swati\_anand01@yahoo.com}

\author [S. Kumar]{Sushil Kumar}
\address {Bharati Vidyapeeth's college of Engineering, Delhi--110063, India}
\email{sushilkumar16n@gmail.com}

\author[V. Ravichandran]{V. Ravichandran}
\address{Department of Mathematics, National Institute of Technology,
Tiruchirappalli--620015, India}
\email{ravic@nitt.edu; vravi68@gmail.com}

\begin{abstract}Theory of differential subordination provides techniques to reduce differential subordination problems into verifying some simple algebraic condition called  admissibility condition. We exploit the first order  differential subordination theory to get several sufficient conditions for function satisfying several differential subordinations  to be a Janowski   function with positive real part.   As  applications, we obtain sufficient conditions for normalized analytic  functions to be Janowski starlike functions.
\end{abstract}


\maketitle

\section{Motivation }
The class of all analytic functions  defined on the unit disk $\mathbb{D} := \left\{z \in \mathbb{C} :|z| < 1\right\}$ that fixes the origin and has derivative 1 at the origin is denoted by $\mathcal{A}$. An analytic function $p$ is   \textit{subordinate} to the analytic function $q$, written $p \prec q$, if $p  = q\circ \omega$ for some analytic function $\omega:\mathbb{D}\to\mathbb{D}$ with $\omega(0) = 0$.  If the function $q$ is univalent in $\mathbb{D}$, then $p\prec q$ if and only if $p(0) = q(0)$ and $p(\mathbb{D}) \subseteq q(\mathbb{D})$.  The class $\mathcal{P}$ consists of Carath\'eodory functions $p:\mathbb{D} \to \mathbb{C}$ of the form $p(z)=1+c_1 z+c_2 z+\cdots$ that  maps the unit disk $\mathbb{D}$ into a region on the right half plane. For arbitrary fixed numbers $A$ and  $B$ satisfying $-1 \leq B < A \leq 1$, denote by  $\mathcal{P}[A,B]$ the class of analytic functions $p\in \mathcal{P}$ satisfy the subordination $p(z) \prec (1+Az)/(1+Bz)$. We call the functions in $\mathcal{P}[A,B]$ as \textit{Janowski   functions with positive real part}.  The class  $\mathcal{S}^*[A,\,B]$ consists  of functions $f \in \mathcal{A}$  such that
${z f'(z)}/{f(z)}\in \mathcal{P}[A,B]$ for $z \in \D.$
The functions in the class $\mathcal{S}^*[A,\,B]$ are called the  \textit{Janowski starlike functions}, introduced by Janowski \cite{Jan}. In particular,  $S^*[1-2 \alpha, -1] = \mathcal{S}^*(\alpha)$ is the class of starlike functions of order $\alpha$, see \cite{goodman 1,Robertson}.

Nunokawa \cite{Nunokawa} proved that if   $1+ zp'(z) \in \mathcal{P}[1,0]$, then  $p\in\mathcal{P}[1,0]$.   In 2007, Ali \emph{et al.} \cite{Ali} determined the conditions on $\beta$ and numbers $A,B,D,E \in [-1,1]$ so that $p\in \mathcal{P}[A,B]$ whenever $1+ \beta zp'(z)$ or $1+ \beta zp'(z)/p(z)$ or $1+ \beta zp'(z)/p^2(z)$ is  in the class $ \mathcal{P}[D,E]$.  In 2018,  authors  \cite{sushil} obtained the sharp lower bound  on $\beta$ so that the function $p(z)$ is subordinate to the functions $e^z$ and $(1+Az)/(1+Bz)$ whenever $1+\beta zp'(z)/p^j(z), \,\,(j = 0,1,2)$ is subordinate to the  functions with positive real part like  $\sqrt{1+z}$, $(1+Az)/(1+Bz)$.
Recently, Ahuja \emph{et al.} \cite{ahuja17} computed sharp estimates for $\beta$ so that a Carath\'{e}odory function is subordinate to a starlike function with positive real part whenever $1+\beta zp'(z)/p^j(z), \,\, (j =0,1,2)$ is subordinate to lemniscate starlike function. For more details, see \cite{bulc, Cho-Bell, Nunokawa,Kanika16BIMS, Tuneski}.
Motivated by work done  in \cite{ahuja17, Ali, Bohra18Hacet, Cho18Turk, Lecko18Rocky,gandhi,  sushil}, by using admissibility condition technique, a condition  on $\beta$ is established  so that $p\in \mathcal{P}[A,B]$ when  $1 + \beta p'(z)/p^k(z)$ with $k \in \mathbb{N} \cup \{0\}$, $ p(z)+ \beta zp'(z)/p^2(z)$, $1+\beta (zp'(z))^2/p^{k}(z)$  and $1/p(z) - \beta zp'(z)/p^k(z)$ for $k \in \mathbb{N} \cup \{0\}$ are in the class $\mathcal{P}[D,E]$. We compute a conditions on $\alpha$ and $\beta$ for $p\in \mathcal{P}[A,B]$ whenever $(1-\alpha)p(z) + \alpha p^2(z)+ \beta zp'(z)/p^k(z)\in \mathcal{P}[D,E]$ for $k = 0,1$ as well.
Additionally, a condition on $\beta$ and $\gamma$ is determined in a  Briot-Bouquet differential type subordination relation:  $p(z) + zp'(z)/(\beta p(z)+\gamma)^2  \in \mathcal{P}[D,E] $ implies $p\in \mathcal{P}[A,B]$.
As an application, we obtained some sufficient conditions for a normalized analytic function $f$ in  $\mathcal{S}^*[A,B]$. Kanas \cite{Kanas06} described the admissibility  condition for the function to map $\mathbb{D}$ on to region bounded by parabola and hyperbola. We prove our result by using the corresponding admissibility conditions for the Janowski functions with positive real part.

\section{Janowski Functions }
Let  $\psi(r,s,t;z)\colon\mathbb{C}^3\times \mathbb{D}\to\mathbb{C}$ be a function and let  $h$ be univalent in $\mathbb{D}$. An analytic function $p$  satisfying the   second-order differential subordination
\begin{equation} \label{1.1}
\psi(p(z), zp'(z), z^2p''(z);z) \prec h(z),
\end{equation}  is known as its \emph{solution}. The univalent function $q$ is  a dominant of the solutions of the differential subordination \eqref{1.1} if $p\prec q$ for all $p$ satisfying \eqref{1.1}. A dominant $\tilde{q}$ which satisfies $\tilde{q} \prec q$ for all dominant $q$ of \eqref{1.1} is known as \emph{best dominant} of \eqref{1.1} and it is unique upto a rotation.
Let $\mathcal{Q}$ be the class consisting of all analytic and injective functions $q$  on $\overline{\mathbb{D}} \setminus \textbf{E}(q)$, where
$\textbf{E}(q) = \{\xi \in \partial\mathbb{D} : \lim_{z \to \xi} q(z) = \infty\}$
such that $q'(\xi) \neq 0$ for $\xi \in \overline{\mathbb{D}}\setminus \textbf{E}(q)$.
Let $\Omega$ be a set in $\mathbb{C}$, $q \in \mathcal{Q}$ and $n$ be a positive integer. The class  $\Psi_n[\Omega,q]$ of admissible functions $\psi: \mathbb{C}^3 \times \mathbb{D} \to \mathbb{C}$ that satisfy the admissibility condition:
\begin{equation}  \label{adm}
\psi(r,s,t;z) \notin \Omega
\end{equation}
whenever\[r = q(\xi), \,s = m ~\xi~ q'(\xi)\,\,\mbox{and}\,\,\RE\left(\frac{t}{s} +1\right) ~\geq ~ m \RE\left(\frac{\xi q''(\xi)}{q'(\xi)} +1\right)\]
for $z \in \mathbb{D}, \xi \in \overline{\mathbb{D}} \setminus \textbf{E}(q)$ and $m \geq n \geq 1$.
In particular, let $\Psi_1[\Omega,q]=\Psi[\Omega,q]$. For more details, see \cite{Kanas,KanasLecko,  Kim, mm1, Seoudy}. For this class $\Psi_n[\Omega,q]$, the following result is well-known.

\begin{theorem}  \cite[Theorem 2.3b, p. 28]{mm} \label{thm 2.3b}
Let the function $\psi \in \Psi_n[\Omega,q]$ with $q(0) =a$. If the function $p \in \mathcal{H}[a,n]$ satisfies
\begin{equation} \label{mm 2.3b}
\psi(p(z),zp'(z),z^2p''(z);z)\in \Omega,
\end{equation}
then $p(z) \prec q(z)$.
\end{theorem}

We begin by describing the class of admissible function $\Psi_n[\Omega,q]$ when $q:\D \to \C$  is the function given by $q(z) = (1+Az)/(1+Bz)$ where  $-1 \leq B < A\leq1$. Note that $q(0) = 1$ and $E(q)\subset\{-1\}$. Clearly, the function $q$ is univalent in $\overline{\D} \setminus \textbf{E}(q)$. Therefore  $q \in \mathcal{Q}$
and the domain $q(\mathbb{D})$ is
\[\Delta\,= q(\mathbb{D}) =  \left\{w \in \mathbb{C}: \left|\frac{w-1}{A-Bw}\right| < 1\right\}.\]
For $ \varsigma = \E$ and $0 \leq \theta \leq 2 \pi$, we have
\[ q(\varsigma) = \frac{1+Ae^{i \theta}}{1+Be^{i \theta}}, \,\,q'(\varsigma) = \frac{A-B}{(1+Be^{i \theta})^2} \,\,\mbox{and}\,\, q''(\varsigma) = \frac{2B(A-B)}{-(1+Be^{i \theta})^3}.\]
and  a simple calculation yields
\[ \RE \left(\frac{\varsigma q''(\varsigma)}{q'(\varsigma)} +1\right) = \frac{(1-B^2)m}{1+B^2 +2B \cos\theta}.\]
Thus we get  the following   \textbf{condition of admissibility}: $\psi(r,s,t;z) \notin \Omega$  whenever $(r,s,t;z) \in \mbox{dom}\, \psi$ and
\begin{equation} \label{coa}
r = \frac{1+Ae^{i \theta}}{1+Be^{i \theta}},  \,\, s = \frac{ m (A-B) \E}{(1+Be^{i \theta})^2} \,\,\mbox{and}\,\,  \RE \left(\frac{t}{s} +1\right) \geq \frac{1-B^2}{1+B^2 +2B \cos\theta}
\end{equation}
where $0 < \theta < 2\pi$ and $m \geq n \geq 1$ and the class of all such functions $\psi$ satisfying  the admissibility condition is denoted by $\Psi(\Omega;A,B)$.

When $q(z)=(1+Az)/(1+Bz)$,  Theorem \ref{thm 2.3b} specializes to the following  first order differential subordination  result:
\begin{theorem}\label{th1.3}
Let $p \in \mathcal{H}[1,\,n]$ with $n \in \mathbb{N}$. Let $\Omega$ be a subset  of $ \mathbb{C}$ and $\psi: \mathbb{C}^2 \times \mathbb{D} \to \mathbb{C}$ with domain $D$ satisfy $\psi(r,s;z) \notin \Omega$ for all  $z \in  \mathbb{D}$, where $r $ and $ s$ are given by \eqref{coa}.
If $(p(z),zp'(z);z) \in D$ and $\psi(p(z),zp'(z);z)\in \Omega$ for $z \in \mathbb{D}$, then
$p\in \mathcal{P}[A,B].$
\end{theorem}

We investigate  functions that naturally arise  in the investigation of univalent functions to be admissible.  In the first result, we show that $\psi(r,s;z) = 1+ \beta {s}/{r^k}$ is an admissible function.

\begin{theorem} \label{Thm 2.1}
Let  $\beta \neq 0$,  $-1 \leq B < A \leq 1$ and  $-1 \leq E < D \leq 1$ satisfy the condition
\begin{itemize}
\item[(i)] $|\beta|(A-B) \geq (D-E) (1+|A|)^k(1+|B|)^{2-k} + |E \beta(A-B)|$; $( k = 0,1,2)$ or
\item[(ii)] $|\beta|(A-B) (1-|B|)^{k-2}\geq (D-E) (1+|A|)^k + |E \beta(A-B)|(1+|B|)^{k-2}$; $ (k > 2).$
\end{itemize} If  $p$ is analytic in $\mathbb{D}$ and
\begin{equation*}
1+ \beta \frac{zp'(z)}{p^k(z)} \in \mathcal{P}[D,E];\quad k \in \mathbb{N} \cup \{0\},
\end{equation*}
then $p\in\mathcal{P}[A,B]$.
\end{theorem}

\begin{proof}
Let  $ \Omega = \{w\in \mathbb{C} : \left|(w-1)/(D-Ew)\right| < 1\}$.
The function $\psi: (\C \setminus \{0\}) \times \C \times \D \to \C$ is defined as
\[\psi(r,s;z) = 1+ \beta \frac{s}{r^k}\]
where $k$ is a non-negative integer. Using the values of  $r,s$ from \eqref{coa}, we have
\begin{equation*}
\psi(r,s;z) = 1+ \beta \frac{m (A-B) e^{i \theta}(1+Be^{i \theta})^{k} }{(1+A e^{i \theta})^k (1+Be^{i \theta})^{2}}.
\end{equation*}
By making use of Theorem \ref{th1.3}, the desired subordination is showed if we prove $\psi \in \Psi[\Omega;A,B]$. For this purpose, set \[\boldsymbol{\chi}(w, \,D, \,E)=\left|\frac{w-1}{D-Ew}\right|.\]

(i) When  $ k = 0,1,2$. A simple calculation gives
\begin{align*}
\left|\boldsymbol{\chi}(\psi(r,s;z), \,D, \,E)\right| &= \left|\frac{\beta m (A-B) e^{i \theta}}{(D-E)(1+Ae^{i \theta})^k(1+Be^{i \theta})^{2-k} -E \beta m e^{i \theta}(A-B)}\right|\\
 &\geq \frac{|\beta| m (A-B)}{(D-E)|(1+Ae^{i \theta})^k| |(1+Be^{i \theta})^{2-k}| + |E \beta m (A-B)|}\\
 & \geq  \frac{|\beta|  m(A-B)}{(D-E)(1+|A|)^k (1+|B|)^{2-k} + m|E \beta (A-B)|} =: \phi(m).
 \end{align*}
Observe  that  the function $\phi(m)$ is an increasing function for $m \geq 1$ by first derivative test. Hence the minimum value of  $\phi(m)$ occurs at $m=1$. Thus, the last inequality becomes
\[\left|\boldsymbol{\chi}(\psi(r,s;z), \,D, \,E)\right|\geq \phi(1)\geq 1\]
if the  inequality  $ |\beta| (A-B) \geq (D-E)(1+|A|)^k(1+|B|)^{2-k} + |E \beta (A-B)|$ holds.
Therefore, $\psi(r,s;z) \notin \Omega$ which implies $\psi \in \Psi(\Omega;A,B)$ and we get the desired  $p \prec q$.

(ii) When  $k >2$, we note that
 \begin{align*}
 \left|\boldsymbol{\chi}(\psi(r,s;z), \,D, \,E)\right| &= \left|\frac{\beta m e^{i \theta}(A-B)(1+Be^{i \theta})^{k-2}}{(D-E)(1+Ae^{i \theta})^k -E \beta m e^{i \theta}(A-B)(1+Be^{i \theta})^{k-2}}\right| \\
  &\geq \frac{|\beta| m (A-B)(1-|B|)^{k-2}}{(D-E)(1+|A|)^k  + |E \beta m (A-B)|(1+|B|)^{k-2}} =: \phi(m).
\end{align*}
As previous case, note that  $\phi(m) \geq \phi(1)$. Hence the last inequality is written as
\[\left|\boldsymbol{\chi}(\psi(r,s;z), \,D, \,E)\right|\geq  1\]
provided $|\beta| (A-B)(1-|B|)^{k-2} \geq (D-E)(1+|A|)^k + |E \beta (A-B)|(1+|B|)^{k-2}$. Therefore,
we get $p \prec q$.
\qedhere
\end{proof}

\begin{remark}
When $k = 0 \, \, \mbox{and}\,\, 1$,  Theorem \ref{Thm 2.1} reduces to \cite[Lemma 2.1, p. 2]{Ali} and \cite[Lemma 2.10, p. 6]{Ali} respectively.
When  $k = 2$ and $ \beta = 1$,  Theorem \ref{Thm 2.1}  simplifies to \cite[Lemma 2.6, p. 5]{Ali}.
\end{remark}

For a positive integer $k$, next theorem gives a conditions on $\beta$ so that the differential subordination  $1+ \beta (zp'(z))^2/p^k(z)\in \mathcal{P}[D,E]$  implies $p\in \mathcal{P}[A,B]$.

\begin{theorem} \label{s^2}
  Suppose  $k$ a non-negative integer, $\beta \neq 0$, $-1 \leq B < A \leq 1$ and $-1 \leq E <0 < D \leq 1$ satisfy either
\begin{itemize}
\item[(i)] for $0 \leq k < 4$,
\begin{align}\label{eq2.51}
|\beta|(A-B)^2 \geq (D-E)(1+|A|)^k(1+|B|)^{4-k}+|E \beta (A-B)^2|,
\end{align}
or
\item[(ii)]for $k \geq 4$,
\begin{align}\label{eq2.52}
|\beta|(A-B)^2(1-|B|)^{k-4} \geq (D-E)(1+|A|)^k + |E \beta(A-B)^2|(1+|B|)^{k-4}.
\end{align}
\end{itemize} If $p$ is analytic in $\mathbb{D}$ and
$1+ \beta (zp'(z))^2/p^k(z) \in \mathcal{P}[D,E]$,
then $p\in\mathcal{P}[A,\,B]$.
\end{theorem}

\begin{proof}
By considering  the  domain $\Omega$ as in Theorem \ref{Thm 2.1} and   the analytic function $\psi(r,s;z) = 1+ \beta {s^2}/{r^k}$ where $k$ is non-negative integer, we need to show $\psi \in \Psi[\Omega,\, A,\,B]$.

(i)  Let $0 \leq k <4$.  In view of (\ref{coa}), we note that

\begin{equation*}
\psi(r,s;z) = 1+ \beta \frac{m^2 (A-B)^2 e^{2i \theta} (1+B \E)^k}{(1+B e^{i \theta})^4 (1+A \E)^k}
\end{equation*}
so that
\begin{align*}
\left|\boldsymbol{\chi}(\psi(r,s;z), \,D, \,E)  \right|&= \left|\frac{\beta m^2 (A-B)^2 e^{2i \theta}}{(D-E)(1+A \E)^k(1+B e^{i \theta})^{4-k} -E \beta m^2 (A-B)^2 e^{2i \theta}}\right|\\
&= \frac{|\beta| m^2(A-B)^2}{|(D-E)(1+A \E)^k(1+B e^{i \theta})^{4-k} -E \beta m^2(A-B)^2 e^{2i \theta}|}\\
& \geq \frac{|\beta| m^2 (A-B)^2}{(D-E)(1+|A|)^k(1+|B|)^{4-k} + m^2|E \beta (A-B)^2|} =: \phi(m).
\end{align*}
A calculation shows that $\phi'(m) > 0$ for $m \geq 1$. Therefore  $\phi(m) \geq \phi(1)$. The last inequality simplifies to
$\left|\boldsymbol{\chi}(\psi(r,s;z), \,D, \,E) \right| \geq  1$ whenever the inequality \eqref{eq2.51}  holds.  As a conclusion it is noted that $\psi(r,s;z) \notin \Omega$. Thus we get the required subordination.

(ii)  Let $k \geq 4$. Proceeding as in (i),  we have
 \begin{equation*}
 \psi(r,s,t;z) = 1+ \frac{ \beta m^2 (A-B)^2 e^{2 i \theta} (1+B\E)^{k-4}}{(1+A\E)^k}.
 \end{equation*}
so that
 \begin{align*}
 \left|\boldsymbol{\chi}(\psi(r,s;z), \,D, \,E)\right| &= \left|\frac{\beta m^2 (A-B)^2 e^{2 i \theta} (1+B\E)^{k-4}}{(D-E)(1+A \E)^k - E \beta m^2 (A-B)^2 (1+B \E)^{k-4} e^{2 i \theta}}\right|\\
  &\geq \frac{|\beta|m^2 (A-B)^2|(1+B\E)^{k-4}|}{|(1+ A\E)^k(D-E)| +m^2|E \beta  (A-B)^2 e^{2 i \theta}||(1+B\E)^{k-4}|}\\
  &\geq \frac{|\beta|m^2 (A-B)^2(1-|B|)^{k-4}}{(1+ |A|)^k(D-E) +m^2|E \beta  (A-B)^2|(1+|B|)^{k-4}}=: \phi(m).
  \end{align*}
A calculation shows that $\phi(m)$ is an increasing function for $m \geq 1$ and thus has minimum value at $m = 1$. As similar  analysis of previous case, we get $p\in\mathcal{P}[A,B]$.
\qedhere
\end{proof}

In \cite{sivaprasad}, a lower bound on $\beta$ is determined such that  $p(z) + \beta zp'(z)/p^2(z) \prec \sqrt{1+z}$ implies $p(z) \prec \sqrt{1+z}$. Recently,  Sharma and Ravichandran \cite{Ravi15JKMS} established  similar type subordination for analytic functions associated to Cardioid. Motivated by this work,  the condition on $\beta$ is computed so that   $p(z) + \beta zp'(z)/p^2(z) \in \mathcal{P}[D,E]$  implies $p\in\mathcal{P}[A,B]$.

\begin{theorem}\label{2.7}
Suppose  $-1\leq B < A \leq 1$ and   $-1 \leq E <D \leq1$ satisfy
\begin{equation}\label{th2.3}\begin{split}
(A-B) (|\beta|(1-|B|)-(1+|A|)^2)& \geq (1+|A|)^2 ((D-E) + |DB-EA|) \\   &\quad{}+|E\beta(A-B)|(1+|B|).\end{split}
\end{equation}
If $p$ is analytic in $\mathbb{D}$ and
$p(z) + \beta  { zp'(z) }/{p^2(z)}  \in \mathcal{P}[D,E]$,
then $p \in \mathcal{P}[A,B]$.

\end{theorem}
\begin{proof}
Consider  the domain $\Omega$ as in Theorem \ref{Thm 2.1}. The analytic function $\psi: \C \setminus \{0\} \times \C \times  \D \to \D$  is defined as $\psi(r,s;z) = r + \beta {s}/{r^2}$.
For required subordination, we need to show $\psi(r,s,t,z) \notin \Omega$. For the values of $r,s$ in \eqref{coa}, we have
\[\psi(r,s;z) =\frac{(1+A\E)^3+\beta m \E(A-B)(1+B \E)}{(1+A \E)^2 (1+B \E)}\]
so that
\begin{align*}
\left|\boldsymbol{\chi}(\psi(r,s;z), \,D, \,E)\right| &= \frac{(A-B)|\beta m(1+B\E)+(1+A\E)^2|}{\splitfrac{|(1+A \E)^2 (D(1+B\E) - E((1+A\E)) -E \beta m \E(A-B)}{(1+B\E)|}}\\
&=\frac{(A-B)|\beta m(1+B\E)+(1+A\E)^2|}{\splitfrac{|(1+A \E)^2 ((D-E) + (DB-EA)\E) -E\beta m \E(A-B)}{(1+B\E)|}}\\
& \geq \frac{(A-B)|\beta m(1+B\E)|-|(1+A\E)^2|}{\splitfrac{|(1+A \E)^2 ((D-E) + (DB-EA)\E)| +|E\beta m \E(A-B)}{(1+B\E)|}}\\
& \geq \frac{(A-B)(|\beta| m(1-|B|)-(1+|A|)^2)}{(1+|A|)^2 ((D-E) + |DB-EA|) +m |E\beta (A-B)|(1+|B|)} \\
&= : \phi(m).
\end{align*}
The function $\phi(m)$ is an increasing for $m \geq 1$. So the function $\phi(m)$ attains its minimum value  at $m=1$. Then
$\left|\boldsymbol{\chi}(\psi(r,s;z), \,D, \,E)\right|\geq \phi(1)\geq 1$
provided the inequality \eqref{th2.3} holds.
By  Theorem \ref{th1.3}, we have  $\psi \in \Psi(\Omega;A,B)$  and this proves the result. \end{proof}

In \cite{Kanika16BIMS}, authors derived condition on $\alpha$ and $\beta$ so that subordination $(1-\alpha) p(z) +  \alpha p^2(z) + \beta zp'(z)/ p^k(z) \prec 1+\frac{4}{3}z+\frac{2}{3}z^2$ $(k=0,\,1)$ implies $p(z) \prec 1+\frac{4}{3}z+\frac{2}{3}z^2$. In view of this work, next two theorems  give a relation between  $\alpha$ and $\beta$ so that  $(1-\alpha) p(z) +  \alpha p^2(z) + \beta zp'(z)/p^k(z) \in \mathcal{P}[D,E]$ (where $k=0,1$) implies $p\in \mathcal{P}[A,B].$

\begin{theorem} \label{ap'}
Let    $-1\leq B < A \leq 1$, $-1 \leq E < 0 <D \leq1$, $\beta \neq 0$ and  $0 \leq \alpha \leq 1$. Assume that
\begin{equation}\label{eQ2.9}
\begin{split}
   (A-B)(|\beta| -(1+|B|) - \alpha(1+|A|)) \geq & (1+|B|)(D(1+|B|)-E(1- \alpha)\\ &(1+|A|))
  -E\alpha(1+|A|)^2+|E \beta (A-B)|.
\end{split}
   \end{equation}
If $p$ is analytic in $\mathbb{D}$ and
$(1-\alpha) p(z) +  \alpha p^2(z) + \beta zp'(z) \in \mathcal{P}[D,E]$, then $p  \in \mathcal{P}[A,B].$
\end{theorem}

\begin{proof}
Consider  the domian $\Omega$ as in Theorem \ref{Thm 2.1}.
The analytic function $\psi:\C^3 \times \D \to \D$ is defined as
$\psi(r,s;z) = (1-\alpha) r +\alpha r^2 + \beta s.$ To show $\psi \in \Psi[\Omega,\, A,\,B]$, it is suffices  to prove $|\boldsymbol{\chi}(\psi(r,s;z), \,D, \,E) | \geq 1$.
It is easy to deduce that
\[\psi(r,s;z)= \frac{(1-\alpha)(1+A \E)(1+B \E) + \alpha (1+A \E)^2 + \beta m (A-B)\E }{(1+B\E)^2}\]
such that

\begin{align*}
\left|\boldsymbol{\chi}(\psi(r,s;z), \,D, \,E)  \right|
   & = \frac{(A-B)|\beta m + (1+B \E)+  \alpha(1+A\E)|}{\splitfrac{|(1+B \E)(D(1+B\E) - E(1- \alpha)(1+A\E))}{-E \alpha (1+A\E)^2 - E\beta m(A-B)\E|}}\\
   & \geq \frac{(A-B)(|\beta|m - (1+|B|) - \alpha (1+|A|))}{\splitfrac{|(1+B \E)||(D(1+B\E) - E(1- \alpha)(1+A\E))|}{+|E \alpha (1+A\E)^2|+| E\beta m(A-B)\E|}}\\
 & \geq \frac{(A-B)(|\beta|m - (1+|B|) - \alpha (1+|A|)}{\splitfrac{(1+|B|)(D(1+|B|) - E(1- \alpha)(1+|A|))}{- E \alpha (1+|A|)^2+| E\beta m(A-B)|}} =: \phi(m).
\end{align*}
Note that  $\phi(m) \geq \phi(1)$ for $m \geq 1$ and therefore $ \left|\boldsymbol{\chi}(\psi(r,s;z), \,D, \,E) \right|  \geq 1$ whenever the  inequality \eqref{eQ2.9}  holds. Thus  $\psi(r,s;z) \notin \Omega$ and  Theorem \ref{th1.3} yields the desired subordination.
\qedhere
\end{proof}


As an implication of  Theorems \ref{s^2}--\ref{ap'}, each of  following  is sufficient condition for function $f \in \mathcal{S}^*[A,B]$:
\begin{itemize}
\item[(a)] \[\frac{zf'(z)}{f(z)}\left(1 + \beta\left( \frac{zf'(z)}{f(z)}\right)^{-1} \left(1+\frac{zf''(z)}{f'(z)} -\frac{zf'(z)}{f(z)}\right)^2 \right) \in \mathcal{P}[D,E]\]
where  $-1 \leq B< A \leq 1$, $-1 \leq E< 0 < D \leq 1$ and $\beta$ satisfies following  inequality
\begin{align*}
|\beta| (A-B)^2 \geq (D-E)(1+|A|)^2(1+|B|)^2+ |E \beta (A-B)^2|,
\end{align*}
\item[(b)]
\[\frac{zf'(z)}{f(z)}\left(1 + \beta\left(\frac{zf'(z)}{f(z)}\right)^{-2} \left(1+\frac{zf''(z)}{f'(z)} -\frac{zf'(z)}{f(z)}\right) \right) \in \mathcal{P}[D,E],\] where  $-1 \leq B< A \leq 1$,  $-1 \leq E< D \leq 1$ and  $\beta$ satisfies an inequality \eqref{th2.3}.

\item[(c)] \[(1- \alpha + \beta) \frac{zf'(z)}{f(z)} +(\alpha - \beta)\left(\frac{zf'(z)}{f(z)}\right)^2 + \beta \frac{zf''(z)}{f'(z)}
       \in \mathcal{P}[D,E]\]
whenever $ \beta \ne 0$, $-1 \leq B < A \leq 1$, $-1 \leq E< 0< D \leq 1$ and the inequality \eqref{eQ2.9} holds.
\end{itemize}

\begin{corollary}
Let $p\in \mathcal{P}$. For $-1\leq B < A \leq 1$ , $-1 \leq E <0 <D \leq1$ and $\beta \ne 0$. We assume that
\begin{equation}\label{2.10}
(A-B) (|\beta|-(1+|B|)) \geq (D-E) + |2BD - E(A+B)|+|E\beta(A-B)| + |DB^2 - EAB|.
\end{equation}
If  $p(z) + \beta zp'(z) \in \mathcal{P}[D,E]$,  then $p\in\mathcal{P}[A,B].$
\end{corollary}

\begin{corollary}
Let $p\in \mathcal{P}$,  $\beta \ne 0$, $-1\leq B < A \leq 1$ and $-1 \leq E < 0 <D \leq1$. We assume the following inequality
\begin{equation*}
\begin{split}
   (A-B)(|\beta| -(1+|B|) - (1+|A|)) \geq & D(1+|B|)^2-E(1+|A|)^2+|E \beta (A-B)|.
\end{split}
\end{equation*}
If $p^2(z) + \beta zp'(z) \in \mathcal{P}[D,E]$, then $ p\in\mathcal{P}[A,B].$
\end{corollary}

\begin{theorem} \label{ap'/p}
 Let $p\in \mathcal{P}$, $-1\leq B < A \leq 1$,  $-1 \leq E < 0 <D \leq1$,  $G=1+|A|$ and  $0 \leq \alpha \leq 1$. Assume that
\begin{align}\label{eQ2.10}
(A-B)(|\beta| -G - \alpha(G)^2(1-|B|)^{-1}) \geq&  (G)(D(1+|B|)-E(1- \alpha)(G))
       \\  \notag & -E \alpha(G)^3 (1-|B|)^{-1}+|E \beta (A-B)|
\end{align}
If $(1-\alpha) p(z) +  \alpha p^2(z) + \beta zp'(z)/p(z)\in\mathcal{P}[D,E]$, then $p\in\mathcal{P}[A,B].$
\end{theorem}

\begin{proof}
By considering $\Omega$ be as in Theorem \ref{Thm 2.1} and the analytic function  $\psi(r,s,t;z) = (1-\alpha) r +\alpha r^2 + \beta s/r$, it is enough to prove $|\boldsymbol{\chi}(\psi(r,s;z)| \geq 1$. Using (\ref{coa}), we have
\begin{align*}
\psi(r,s;z) 
& = \frac{(1-\alpha)(1+A \E)^2(1+B \E) + \alpha (1+A \E)^3 + \beta m  (A-B)(1+B\E)\E}{(1+A \E)(1+B\E)^2}.
\end{align*}
A simple computation yields
\begin{align*}
\left|\boldsymbol{\chi}(\psi(r,s;z)\right| 
    & = \left|\frac{\splitfrac{(1+A \E)^2(1+B \E) +\alpha(1+A \E)^2(A-B) \E +}{ \beta m  (A-B)(1+B\E)\E - (1+A\E)(1+B \E)^2}}{\splitfrac{D (1+A\E)(1+B \E)^2 - E((1-\alpha)(1+A \E)^2(1+B \E) +}{ \alpha (1+A \E)^3 + \beta m (A-B)(1+B\E) \E)}}\right|\\
   & = \frac{(A-B) |\beta m + (1+A \E) + \alpha (1+A\E)^2 (1+B \E)^{-1}|}{\splitfrac{|(1+A\E)(D(1+B\E) - E(1- \alpha)(1+A\E)) -}{ E \alpha (1+A\E)^3(1+B \E)^{-1} - E \beta m \E (A-B)|}}\\
   & \geq \frac{(A-B)(|\beta| m - (1+|A|) - \alpha(1+|A|)^2(1-|B|)^{-1})}{\splitfrac{|(1+A\E)(D(1+B\E) - E(1- \alpha)(1+A\E))| +}{| E\alpha (1+A\E)^3(1+B \E)^{-1}| +|E \beta \E m (A-B)|}}\\
   & \geq \frac{(A-B)(|\beta| m - (1+|A|) - \alpha(1+|A|)^2(1-|B|)^{-1})}{\splitfrac{(1+|A|)(D(1+|B|) - E(1-\alpha)(1+|A|)) -}{E \alpha (1+|A|)^3(1-|B|)^{-1} +m|E \beta (A-B)|}}=:\phi(m).
\end{align*}
It is  observed that  $\phi'(m) >0 $ for all $m \geq 1$. As computation done in the previous theorem, we get the required subordination result.
\qedhere
\end{proof}
\begin{corollary}
Let $p\in \mathcal{P}$,  $\beta \ne 0$, $-1\leq B < A \leq 1$ and $-1 \leq E< 0 <D \leq1$. Suppose that
\begin{equation}\label{2.15}
(A-B) (|\beta|-(1+|A|)) \geq (D-E) + |D(B+A) - 2 EA|+|E\beta(A-B)| + |DBA - EA^2|.
\end{equation}
If $p(z) + \beta { zp'(z) }/{p(z)} \in \mathcal{P}[D,E]$, then
$p\in\mathcal{P}[A,B].$
\end{corollary}

For a positive integer $k$, the condition on $\beta$ is determined so that  $(1/p(z)) - \beta zp'(z)/p^k(z) \in \mathcal{P}[D,E]$ implies $p\in \mathcal{P}[A,B]$.

\begin{theorem} \label{p^k}
Let $p\in \mathcal{P}$,  $-1 \leq B < A \leq 1$ and $-1 \leq E < 0 < D \leq 1$. Then $p\in \mathcal{P}[A,B]$
for each of the following subordination conditions:
\begin{itemize}
\item[(a)]  $(1/p(z)) - \beta zp'(z)\in\mathcal{P}[D,E]$ where $\beta$ satisfies
\begin{equation}\label{eQL1}
\begin{split}
(A-B)(|\beta|(1-|A|) - (1+|B|)^2) \geq & (1+|B|)^2(D(1+|A|) - E(1+|B|))  \\&+|E \beta (A-B)|(1+|A|)
\end{split}
\end{equation}
\item[(b)]  $(1/p(z)) - \beta zp'(z)/p^k(z) \in\mathcal{P}[D,E]$ for $k = 1,2$ where $\beta$ satisfies
\begin{equation}\label{eQL2}
\begin{split}
(A-B)(|\beta|-(1+|A|)^{k-1}(1+|B|)^{2-k} )\geq &(1+|A|)^{k-1}(1+|B|)^{2-k}(D(1+|A|)\\ & - E(1+|B|)) +|E \beta (A-B)|.
\end{split}
\end{equation}
\item[(c)]  $(1/p(z)) - \beta zp'(z)/p^k(z) \in\mathcal{P}[D,E]$ for $k >2$ where $\beta$ satisfies
\begin{equation}\label{eQL3}
\begin{split}
(A-B)(|\beta|(1-|B|)^{k-2}-(1+|A|)^{k-1}) \geq &  (1+|A|)^{k-1}(D(1+|A|) - E(1+|B|))\\ & +|E \beta (A-B)|(1+|B|)^{k-2}.
\end{split}
\end{equation}
\end{itemize}
\end{theorem}

\begin{proof}
For $k = 0, 1,2,3, \ldots$, let  $\Omega$ as in Theorem \ref{Thm 2.1} and the  function $\psi$ be defined as
\begin{equation*}
\psi(r,s;z) = \frac{1}{r} - \beta \frac{s}{r^k}.
\end{equation*}
In view of (\ref{coa}), the function $\psi$ takes the following shape:
\begin{equation}
\psi(r,s;z) = \frac{1+B\E}{1+A \E} - \frac{\beta m \E (A-B)(1+B \E)^k}{(1+B \E)^2 (1+A \E)^k}.
\end{equation}

(a) For $k = 0$, we have
\begin{align*}
\left|\boldsymbol{\chi}(\psi(r,s;z), \,D, \,E)\right|
 &=\frac{|(A-B) ((1+B\E)^2 + \beta m(1+A \E))|}{\splitfrac{|(1+B\E)^2(D(1+A \E) - E(1+B\E)) + E \beta m (A-B)\E}{(1+A\E)|}}\\
 & \geq \frac{(A-B)(|\beta|m|1+A\E| - |(1+B \E)^2|)}{\splitfrac{|(1+B\E)^2(D(1+A \E) - E(1+B\E))|+ |E \beta m  (A-B)\E}{(1+A\E)|}}\\
 &\geq \frac{(A-B)(|\beta|m(1-|A|) - (1+|B|)^2)}{\splitfrac{|(1+B\E)^2| |(D(1+A \E) - E(1+B\E))|+ |E \beta m (A-B)}{(1+A\E)|}}\\
 & \geq \frac{(A-B)(|\beta|m(1-|A|) - (1+|B|)^2)}{\splitfrac{(1+|B|)^2(|D(1+A \E)| + |E(1+B\E)|)+ |E \beta m (A-B)|}{|(1+A\E)|}}\\
 &\geq \frac{(A-B)(|\beta|m (1-|A|) - (1+|B|)^2)}{(1+|B|)^2(|D|(1+|A|)| + |E|(1+|B|))+ |E \beta m (A-B)|(1+|A|)}\\
 &=: \phi(m).
 \end{align*}
Using first derivative test we note that $\phi$ is an increasing function for $ m \geq 1$. Thus the function $\phi(m)$ has minimum value at $m =1$. Therefore  $ \left|\boldsymbol{\chi}(\psi(r,s;z), \,D, \,E) \right|  \geq 1$ whenever the  inequality \eqref{eQL1} holds. Thus Theorem \ref{th1.3} complete the desired proof.
Part (b) and (c) can be proved as  part (a). We are  omitting further details here.
\qedhere
\end{proof}

Let $ \beta \ne 0$, $-1 \leq B< A \leq 1$ and $-1 \leq E< 0 < D \leq 1$. If  one of the   following  subordination  holds for $f \in \mathcal{A}$:
\begin{itemize}


\item[(i)] For $(A-B)(|\beta|-(1+|B|)) \geq (1+|B|)(D(1+|A|) - E(1+|B|)) +|E \beta (A-B)|$,
\[\left( \frac{zf'(z)}{f(z)}\right)^{-1}\left(- \beta \frac{zf'(z)}{f(z)}\left(1+\frac{zf''(z)}{f'(z)} -\frac{zf'(z)}{f(z)}\right) \right)\in \mathcal{P}[D,E];\]

\item[(ii)] For $(A-B)(|\beta|-(1+|A|)) \geq (1+|A|)(D(1+|A|) - E(1+|B|)) +|E \beta (A-B)|$,
\[\left( \frac{zf'(z)}{f(z)}\right)^{-1}\left(1- \beta \left(1+\frac{zf''(z)}{f'(z)} -\frac{zf'(z)}{f(z)}\right) \right)\in \mathcal{P}[D,E];\]
\end{itemize}
then $f \in \mathcal{S}^*[A,B]$.

Motivated by the work in \cite{Ali BB}, we obtain the conditions on $A,B,D,E$ for  a   general  Briot-Bouquet differential subordination in the following theorem.

\begin{theorem} \label{bg}
Let $-1\leq B < A \leq 1$,  $-1 \leq E <D \leq1$ and $\beta \gamma >0$ satisfy
\begin{align}\label{BBE}
(A-B)((1-|B|) - (\beta(1+|A|)+ \gamma(1+|B|))^2) & \geq (\beta(1+|A|) + \gamma(1+|B|))^2   (D-E \notag \\
   & +|DB-EA|)+ |E|(A-B)(1+|B|).
\end{align}
If  $p\in \mathcal{P}$ and  $p(z) +  {zp'(z)}/{(\beta p(z)+ \gamma)^2} \in \mathcal{P}[D,E],
$
then $p\in \mathcal{P}[A,B]$.
\end{theorem}

\begin{proof}
Let  $\Omega$ be defined  as in Theorem \ref{Thm 2.1}. Consider the analytic function
 \[ \psi(r,s;z) = r + \frac{s}{(\beta r + \gamma)^2}.\]
The required subordination is obtained if we show $\psi \in \Psi[\Omega,\, A,\,B]$ by making use of  Theorem \ref{th1.3}.
Using \eqref{coa},  the function $ \psi(r,s;z)$ takes the following form
 \begin{equation*}
 \psi(r,s;z)  = \frac{(1+A \E ) (\beta (1+A \E)+ \gamma(1+B \E))^2+ m (A-B)(1+B \E)e^{i \theta} }{(1+B \E) (\beta (1+A \E)+ \gamma(1+B \E))^2}
 \end{equation*}

 so that
\begin{align*}
 \left|\boldsymbol{\chi}(\psi(r,s;z), \,D, \,E) \right|
 &=\frac{|(A-B) ((\beta (1+A \E)+ \gamma(1+B \E))^2+m(1+B \E))|}{\splitfrac{|(D(1+B\E) - E((1+A \E))(\beta (1+A \E)+ \gamma(1+B \E))^2} {-E m \E(A-B)(1+B \E)|}}\\
 & \geq  \frac{(A-B)(m |1+B\E| - |(\beta(1+A \E) + \gamma(1+B\E))^2|)}{\splitfrac{|(D-E)(\beta (1+A \E)+ \gamma(1+B \E))^2| +|\E(DB-EA)}{(\beta (1+A \E)+ \gamma(1+B \E))^2| + |E m \E(A-B)(1+B \E)|}}\\
 & \geq \frac{(A-B)( m (1-|B|) - (|\beta(1+A \E)|+|\gamma (1+B \E)|)^2)}{\splitfrac{|(D-E)(\beta (1+A \E)+ \gamma(1+B \E))^2| +|(DB-EA)}{(\beta (1+A \E)+ \gamma(1+B \E))^2| + |E m (A-B)(1+B \E)|}}\\
 & \geq \frac{(A-B)(m (1-|B|) - (\beta(1+|A|)+\gamma (1+|B|))^2)}{\splitfrac{(D-E)(\beta (1+|A|)+ \gamma(1+|B|))^2 +|(DB-EA)}{(\beta (1+|A|)+ \gamma(1+|B|))^2| + m |E (A-B)|(1+|B|)}}  =: \phi(m).
 \end{align*}
A computation shows that $\phi'(m) > 0$. Thus for $m \geq 1$, $\phi(m) \geq \phi(1)$ and therefore $ \left|\boldsymbol{\chi}(\psi(r,s;z), \,D, \,E) \right|  \geq 1$ whenever the  inequality \eqref{BBE} holds.
This implies that $\psi(r,s;z) \notin \Omega$. Hence the desired subordination is obtained.
\qedhere
\end{proof}
%

 \end{document}